\documentclass[reqno, 12pt]{amsart}
\usepackage{verbatim}
\usepackage{amssymb}
\usepackage{enumerate}
\usepackage[active]{srcltx}
\numberwithin{equation}{section}

\usepackage{t1enc}

\usepackage[utf8x]{inputenc}

\newtheorem{theorem}{Theorem}[section]
\newtheorem{proposition}[theorem]{Proposition}
\newtheorem{lemma}[theorem]{Lemma}
\newtheorem{corollary}[theorem]{Corollary}

\newtheorem{claim}{Claim}[theorem]
\newtheorem{problem}{Problem}[section]

\theoremstyle{definition}

\newtheorem{definition}[theorem]{Definition}

\newcommand{\mc}[1]{\mathcal{#1}}
\newcommand{\mbb}[1]{\mathbb{#1}}

\newcommand{\mf}[1]{\mathfrak{#1}}

\newcommand{\setm}{\setminus}
\newcommand{\empt}{\emptyset}
\newcommand{\subs}{\subset}
\newcommand{\dom}{\operatorname{dom}}

\def\<{\left\langle}
\def\>{\right\rangle}
\def\br#1;#2;{\bigl[ {#1} \bigr]^ {#2} }

\newcommand{\rcp}{\operatorname{RC}^+}
\newcommand{\clp}{\operatorname{\mc F}^+}
\newcommand{\MAC}{{\rm MA(\mbox{\it countable})}}

\newcommand{\Pjuh}[1]{\mc J_{#1}}
\newcommand{\pjuh}[1]{J_{#1}}

\author[I. Juh\'asz]{Istv\'an Juh\'asz}
\address
      { Alfréd Rényi Institute of Mathematics, Hungarian Academy of Sciences
}
\email{juhasz@renyi.hu}

\author[L. Soukup]{Lajos Soukup}
\address
      { Alfréd Rényi Institute of Mathematics, Hungarian Academy of Sciences}
\email{soukup@renyi.hu}

\author[Z. Szentmikl\'ossy]{Zolt\'an Szentmikl\'ossy}
\address{E\"otv\"os University of Budapest}
\email{szentmiklossyz@gmail.com
}

\title
   {Anti-Urysohn spaces}

\date{\today}
\thanks{The research on and preparation of this paper was
supported by  OTKA grant no. K 113047.}
\begin{document}

\begin{abstract}
All spaces are assumed to be infinite Hausdorff spaces.
We call a space  {\em anti-Urysohn $($AU in short$)$} iff
any two non-emty regular closed sets in it intersect.
We prove that
\begin{itemize}
\item for every infinite cardinal ${\kappa}$
there is a space of  size ${\kappa}$ in which
fewer than $cf({\kappa})$ many non-empty regular closed sets always intersect;
\item there is  a locally countable
AU space of size $\kappa$ iff $\omega \le \kappa \le 2^{\mf c}$.
\end{itemize}

\smallskip

A  space with at least two non-isolated points is called {\em strongly anti-Urysohn $($SAU in short$)$}
iff  any two infinite  closed sets in it intersect.
We prove that
\begin{itemize}
\item if $X$ is any SAU space then
$ \mf s\le |X|\le 2^{2^{\mf c}}$;

\item if $\mf r=\mf c$ then there is a
separable, crowded, locally countable,  SAU space of
cardinality $\mf c$;
\item if $\lambda > \omega$ Cohen reals are added to any ground model
then in the extension there are SAU spaces of size $\kappa$ for
all $\kappa \in [\omega_1,\lambda]$;
\item if GCH holds and  $\kappa \le\lambda$ are uncountable regular
cardinals then in some CCC generic extension we have
$\mf s={\kappa}$, $\,\mf c={\lambda}$, and for every cardinal
${\mu}\in [\mf s, \mf c]$ there is an   SAU space of cardinality ${\mu}$.
\end{itemize}

The questions if SAU spaces exist in ZFC or if SAU spaces of cardinality $> \mf c$ 
can exist remain open.

\end{abstract}

\maketitle

\section{Introduction}

In this paper ``space'' means ``infinite Hausdorff topological space''.

The space $X$ is called {\em anti-Urysohn}   ({\em AU}, in short) iff
$A\cap B\ne \empt$ for any $A,B\in \rcp(X)$, where $\rcp(X)$ denotes the family of non-empty
regular closed sets in $X$.

We call the space $X$ {\em strongly anti-Urysohn} ({\em SAU}, in short)
iff  $|X'|>1$, that is $X$ has at least two non-isolated points,
and $A\cap B\ne \empt$ for any $A,B\in \ \clp(X)$, where
$\clp(X)$ denotes the family of {\em infinite} closed subsets of $X$.
Clearly, AU spaces are crowded and a crowded SAU space is AU. Our original intention was to include
crowdedness in the definition of SAU spaces. However, we changed our minds after
we realized that it seems to be just as hard to construct them with the weaker property of having
at least two non-isolated points.

What led us to consider AU spaces was not just idle curiosity. Co-operating via correspondence with Alan Dow,
we have recently arrived at the result that in the Cohen model any separable and sequentially compact
Urysohn space has cardinality $\le \mathfrak{c}$. (This result will be published elsewhere.)
The natural question if this holds for all (Hausdorff) spaces, however, remained open.
When trying to find a ZFC counterexample, it was natural to look for spaces that are
as much non-Urysohn as possible.

Actually, a countable AU space, under a different name, had been constructed
by W. Gustin in \cite{Gu} a long time ago, as a simple(r) example
of a countable connected Hausdorff space. (It is obvious that AU spaces are connected.)
The first example of a countable connected Hausdorff space was constructed by
Urysohn in \cite{Ur}, but his construction is
extremely long and complicated: just the description of his space takes up three pages.
(We have no idea if Urysohn's example is AU or not.)
A much simpler example was obtained by
Gustin in \cite{Gu} where the following was proved:

\cite[Theorem 4.2]{Gu}
There is a countably infinite Hausdorff space $X$
such that no two distinct points in $X$ have disjoint closed neighbourhoods (i.e.
$X$ is  AU).

An even simpler construction of
a countable connected Hausdorff space, which takes up only one page, was published by
Bing in  \cite{Bi}. This is also presented as example 6.1.6 in Engelking's book \cite{En}.
Bing's example also turns out to be AU.

In contrast to this, we are not aware of any earlier appearance of SAU spaces. In fact,
we admit that when we first considered them we did not think that they could exist.

Our notation and terminology is standard. In set theory we follow \cite{Ku} and
in topology \cite{En}.

\section{Existence of anti-Urysohn spaces}

In this section we show that for every infinite cardinal $\kappa$ there
is an AU space of cardinality $\kappa$. Actually, we prove much more that is
new and of interest even for the case $\kappa = \omega$.

To do that we need the following somewhat technical lemma that provides a general method for
constructing AU spaces.

\begin{lemma}\label{lm:large_SA}
Assume that ${\kappa}$ is an infinite cardinal and $X$ is a space with
$X\cap {\kappa}=\empt$, moreover
$\{K_{\alpha}:{\alpha}<{\kappa}\}$ are pairwise
disjoint non-empty compact subsets of $X$ such that
\begin{enumerate}[(1)]
 \item if $\,a\subs {\kappa}$ is cofinal then
$\,\bigcup\nolimits_{{\alpha}\in a}K_{\alpha}$ is dense in $X$;

\smallskip

\item $Y=X\setm \bigcup\nolimits_{{\alpha} < \kappa}K_{\alpha}$ is also dense in $X$.
\end{enumerate}
\smallskip
Define the topology $\varrho$ on $Z = Y\cup {\kappa}$  as follows:
\begin{itemize}
\item for  $y \in Y$ the family $\{U \cap Y : y \in U \in \tau(X) \}$,
\smallskip
\item for ${\alpha}\in {\kappa}$ the family
\begin{align*}
\Big\{\{{\alpha}\}\cup (W\cap Y) : K_{\alpha} \subs W \in \tau(X) \Big\}
\end{align*}
\end{itemize}
is a $\varrho$-neighbourhood base. Then
\begin{enumerate}[(i)]
\item $\varrho$ is a Hausdorff topology on $Z$,
\item if  $V \in \varrho$ is non-empty then  $\,\overline{V}^\varrho$
includes a final segment of ${\kappa}$.
\end{enumerate}
\end{lemma}

\begin{proof}[Proof of Lemma \ref{lm:large_SA}]
It is straight forward to see that the above definition of $\varrho$ is correct.
That $\varrho$ is Hausdorff follows from the fact that any two disjoint
compact sets have disjoint neighbourhoods in $X$.

It is obvious from the definition that the subspace topology of $Y$ inherited from $X$ is
the same as $\varrho \upharpoonright Y$, moreover $Y$ is dense open in $\langle Z,\varrho \rangle$.
Thus it suffices to show (ii) for $V \in \varrho$ with $V \subs Y$.
But every $\varrho$-open set $V \subs Y$ is of the form $U \cap Y $ with $U \in \tau(X)$.
Now, if $V$ is non-empty then (1) implies that
\begin{align*}
 I=\{{\alpha}<{\kappa}:  K_{\alpha}\cap U \ne \empt\}
\end{align*}
contains a final segment of ${\kappa}$. But if ${\alpha} \in I$ and
$K_\alpha \subs W \in \tau(X)$ then $W \cap U \supset K_{\alpha}\cap U \ne \empt$, hence
$W \cap V = (W \cap Y) \cap V = W \cap U \cap Y \ne \empt$ as well  because, by (2),
$Y$ is dense in $X$. Thus we have ${\alpha}\in \overline {V}^\varrho$ for all $\alpha \in I$,
which completes the proof of (ii).
\end{proof}

Now we shall present two applications of lemma \ref{lm:large_SA}.

\begin{theorem}\label{tm:any}
For any infinite cardinal ${\kappa}$ there is a(n AU) space $Z$
such that
$|Z|={\kappa}$, $\,d(Z) = \log \kappa$,  and
\begin{align*}
\text{$\bigcap \mc A\ne \empt$ whenever
$\mc A\in \br \rcp(X);<cf({\kappa});$.}
\end{align*}
\end{theorem}

\begin{proof}
[Proof of theorem \ref{tm:any}]
Consider in the Cantor cube $\mathbb{C}_\kappa = \{0,1\}^\kappa$ the pairwise
disjoint non-empty compact subsets
$$K_\alpha = \{x \in \mathbb{C}_\kappa : x(\alpha) = 1 \mbox{ and } x(\beta) = 0 \mbox{ for all } \alpha < \beta < \kappa \}.$$
It is well known that $d(\mathbb{C}_\kappa) = \log \kappa$ and we leave it to the reader to check that
the standard proof of this fact (see e.g. \cite{Ju}) yields a dense set $Y \subs \mathbb{C}_\kappa$ with $|Y| = \log \kappa$
such that $Y \cap \bigcup_{\alpha < \kappa}K_\alpha = \empt$.

Now, it is obvious that we may apply lemma \ref{lm:large_SA} to the subspace $X = \bigcup_{\alpha < \kappa}K_\alpha\, \cup\, Y$
of $\mathbb{C}_\kappa$ to obtain the required space on $Z = Y \cup \kappa$.
\end{proof}

In the case $\kappa = \omega$ the countable AU, hence connected, space we obtain from theorem \ref{tm:any}
has the stronger property that any intersection of finitely many
non-empty regular closed sets is non-empty. We think that its construction is at least as simple as
Bing's in \cite{Bi}. In any case, it is certainly stronger because,
as is easily checked, both Gustin's and Bing's countable AU spaces contain
three non-empty regular closed sets whose intersection is empty.

Our next application of lemma \ref{lm:large_SA} will enable us to construct large AU spaces
that are locally small.
Before formulating it we recall that the dispersion character
$\Delta(X)$ of a space $X$ is the
smallest size of a non-empty open set in $X$.

\begin{theorem}\label{tm:large_SA}
For any infinite cardinal ${\kappa}$ there is an AU space $Z$
with a closed discrete subset $D\subs Z$
such that
$|D|=|Z|=\Delta(Z)=d(Z)={\kappa}$  and
\begin{align}\label{eq:large_SA0}
\text{$|D\setm A|<{\kappa}$ for each
$A\in \rcp(X)$.}
\end{align}
Moreover, there is a family
$\{U_{\alpha}:{\alpha}<{\kappa}\}$ of pairwise disjoint  open
sets in $Z \setm D$
such that  every point $z\in Z$ has a neighborhood $W_z$ for which
\begin{align}\label{eq:large_SA2}
\{{\alpha}<{\kappa}: U_{\alpha}\cap W_z\ne \empt\}
 \text{ is bounded in ${\kappa}$}.
\end{align}

\end{theorem}

\begin{proof}[Proof of theorem \ref{tm:large_SA}]
Let  $\mbb E_{\kappa}$ be the  product space $L({\kappa})^{\kappa}$,
where $L({\kappa})$ denotes $\kappa$ with the usual ordinal topology.

For any ${\alpha}<{\kappa}$ we let
\begin{align*}
K_{\alpha}=\{p\in \mbb E_{\kappa}:\ &
 p({\zeta})\le {\alpha} \text{ for all ${\zeta}<{\alpha}$},
\\&p({\alpha})=1 \land  p({\beta})=0
\text{ for all }
{\alpha < \beta} < {\kappa}\setm ({\alpha}+1)\}.
\end{align*}
Then the $K_{\alpha}$ are pairwise disjoint compact subsets of $\mbb E_{\kappa}$.
Put $K_\kappa =  \bigcup_{\alpha < \kappa}K_\alpha$.
Then clearly $\mbb E_{\kappa}\setm K_{\kappa}$ is dense
in $\mbb E_{\kappa}$,
$\Delta(\mbb E_{\kappa}\setm K_{\kappa})=2^{\kappa}$,
and $w(E_{\kappa})={\kappa}$, hence there is a dense set
$Y\subs \mbb E_{\kappa}\setm K_{\kappa}$ such that
$|Y|=\Delta(Y)={\kappa}$.

We may then apply  lemma \ref{lm:large_SA} to the space $X=Y\cup K_{\kappa}$ to obtain the space
$Z=\<Y\cup {\kappa},\varrho\>$ with the closed discrete set $D={\kappa}$.
Clearly, we have $|D|=|Z|=\Delta(Z)=d(Z)={\kappa}$  and property (\ref{lm:large_SA}) holds.

Let us now define
$$U_{\alpha}=\{p\in Y: p(0)={\alpha}+1\}$$ for
${\alpha}<{\kappa}$.
Since the singleton $\{{\alpha}+1\}$ is open in $L({\kappa})$, we have
$U_{\alpha} \in \varrho$, and clearly $\alpha \ne \beta$ implies $U_\alpha \cap U_\beta = \empt$.

Now, if $y\in Y$  then
$$W_y=\{p\in Y: p(0)\le y(0)\} \in \varrho$$
is a neighborhood of $y$ that witnesses \eqref{eq:large_SA2}
because      $W_y\cap U_{\beta}=\empt$
for ${\beta}\ge y(0)$.

If, on the other hand, $\alpha \in {\kappa}$, then $G_\alpha = \{p \in \mathbb{E}_\kappa : p(0) \le \alpha + 1\}$
is an open subset of $\mathbb{E}_\kappa$ with $K_\alpha \subs G_\alpha$, hence
\begin{align*}
W_\alpha=\{\alpha\}\cup \{p\in Y: p(0)\le {\alpha}+1\} \in \varrho
\end{align*}
is a neighborhood of $\alpha$ in $Z$ with
$W_\alpha \cap U_{\beta}=\empt$ for all ${\beta}>{\alpha}$.
\end{proof}

We say that a space $X$ is {\em locally $\kappa$} if every point of $X$
has a neighbourhood of cardinality $\le\, \kappa$.
The following easy result yields an upper bound for the cardinality of a locally $\kappa$
AU space. It will also be used in section 3 for SAU spaces.

\begin{theorem}\label{tm:bound_loc_countable_SA}
Any  locally $\kappa$ space $X$ contains an infinite clopen subset
$Y$ of cardinality $|Y| \le 2^{2^\kappa}$.
\end{theorem}

\begin{proof}
Since $X$ is Hausdorff, we have $|\overline A|\le 2^{2^\kappa}$ for all $A\in \br X;{\le \kappa};$.
We may fix for every point $p\in X$ a  neighborhood $U_p$ of size $\le\, \kappa$.
A very simple closure procedure
then yields an infinite  subset $Y\subs X$ of cardinality $\le 2^{2^\kappa}$
such that
\begin{enumerate}[(a)]
\item $U_p\subs Y$ for all $p \in Y$,
\item $\overline{A}\subs Y$ for all $A\in \br Y;{\le \kappa};$.
\end{enumerate}
Then (a) implies that $Y$ is open and (b) implies that $Y$ is closed because
$t(X) \le \kappa$.
\end{proof}

It is immediate from theorem \ref{tm:bound_loc_countable_SA} that any space which is locally $\kappa$
and connected, in particular AU,  has cardinality  $\le 2^{2^\kappa}$.
The following result implies that this upper bound is sharp: For every $\kappa$ there is a locally
$\kappa$ AU space of cardinality   $ 2^{2^\kappa}$.

\begin{theorem}\label{tm:AU2}
For every infinite cardinal ${\kappa}$ there is a locally $\kappa$  space $X$
with a  closed discrete subset $D$ such that
\begin{enumerate}[(i)]
\item $|X|=2^{2^{\kappa}}$ and  $d(X)=|D|={\kappa}$,
\item $|D\setm A|<{\kappa}$ holds
for any $A\in \rcp(X)$.
\end{enumerate}
In particular, then $\bigcap \mc A\ne \empt$ whenever
$\mc A\in \br \rcp(X);<cf({\kappa});$.
\end{theorem}

\begin{proof}[Proof of Theorem \ref{tm:AU2} ]
By  theorem  \ref{tm:large_SA}
there is a space $Z$ with a closed discrete $D\subs Z$ such that
$|D|=|Z|=\Delta(Z)=d(Z)={\kappa}$ and
\begin{align*}
\text{$|D\setm A|<{\kappa}$ for all
$ A\in \rcp(Z)$,}
\end{align*}
moreover there are pairwise disjoint  open sets
$\{U_{\alpha}:{\alpha}<{\kappa}\}$ in $Z \setm D$ so that
every point  $z\in Z$ has a neighborhood $W_z$  which meets $U_\alpha$
only for boundedly many $\alpha < \kappa$.

The underlying set of our space is
\begin{align*}
X=Z\cup S({\kappa}),
\end{align*}
where $S({\kappa})$ is the set of all uniform ultrafilters on
${\kappa}$.  (Of course, we may assume that $Z \cap S({\kappa}) = \empt$.)
So we have $|X| = |S(\kappa)| = 2^{2^{\kappa}}$.

Next we define the topology $\tau$ on $X$ with the  following stipulations:
\begin{enumerate}[(i)]
 \item $Z\in \tau$ and the subspace topology
of $Z$ inherited from $X$ is the original topology of $Z$;
\item for any uniform ultrafilter $x\in S({\kappa})$
the family
\begin{align*}
 \mc U_x=\Big\{\{x\}\cup\bigcup\nolimits_{{\alpha}\in a} U_{\alpha}\ :\
a\in x\Big\}.
\end{align*}
is a $\tau$-neighborhood base of $x$.
\end{enumerate}

It is obvious from this definition that $Z$ is a dense open subspace of $X$, moreover
$D$ remains  a closed discrete set in  $X$. This immediately implies (i), while (ii)
follows because if  $A \in \rcp(X)$ then $A \cap Z \in \rcp(Z)$. The only thing
that is left to show is the Hausdorffness of $X$.

Since $Z$ is Hausdorff and open in $X$, it is obvious that any two points of $Z$
can be separated in $X$. If $\{x,y\} \in [S(\kappa)]^2$ then there are $a \in x$ and
$b \in y$ with $a \cap b = \empt$, hence

\begin{align*}
 \{x\}\cup\bigcup\nolimits_{\alpha\in a} U_\alpha
\text{\quad and \quad}
 \{y\}\cup\bigcup\nolimits_{\beta\in b} U_\beta
\end{align*}
are disjoint neighborhoods of $x$ and $y$ in $X$.

Finally, assume that $z\in Z$ and $x \in S({\kappa})$.
Then, by \eqref{eq:large_SA2},
there is ${\xi}\in {\kappa}$ such that
$W_z\cap U_{\zeta}=\empt$ for all ${\xi}\le {\zeta}<{\kappa}$.
But we may pick $a\in x$ with $a\cap {\xi}=\empt$, and
then $W_z$ and
\begin{align*}
 \{x\}\cup\bigcup\nolimits_{{\zeta}\in a} U_\zeta
\end{align*}
are disjoint neighborhoods of $z$ and $x$.
\end{proof}

Since in the above construction  $S({\kappa})$ is clearly closed discrete in
$X$, we actually get the following result.

\begin{corollary}
Given $\kappa \ge \omega$, for every cardinal $\lambda \le 2^{2^{\kappa}}$ there is a locally $\kappa$
AU space of cardinality $\lambda$. In particular, for every infinite cardinal $\lambda \le 2^\mathfrak{c}$
there is a locally countable AU space of cardinality $\lambda$.
\end{corollary}

\section{Existence of strongly anti-Urysohn spaces}

We will see later in this section that, at least consistently,  strongly anti-Urysohn (SAU) spaces
exist. However, in strong contrast to the case of AU spaces, there are both lower and upper bounds
for their possible cardinalities. Before establishing these bounds, in the following theorem we
collect some simple properties of SAU spaces.

\begin{theorem}\label{tm:sAU-basic}
Let $X$ be any SAU space. Then
\begin{enumerate}[(1)]
 \item $X$ is countably compact;
\item every compact subset of $X$ is finite;
\item any infinite closed set $F \subs X$ is uncountable; hence
$A\in \br X;{\omega};$ implies $|A'|>{\omega}$;\smallskip
\item $\clp(X)$ is closed under countable intersections.
\end{enumerate}
\end{theorem}

\begin{proof}
(1) An infinite closed discrete set breaks into two disjoint infinite closed sets.

\smallskip\noindent (2)
Assume that $F\subs X$ is compact and let $p_0$ and $p_1$ be two different accumulation points
of $X$ with disjoint neighborhoods $U_0$ and $U_1$, respectively.
Then  the compact set $F\setm U_0$ and the point $p_0$ have disjoint open
neighbourhoods: $F\setm U_0\subs V$, $p_0\in W$ and $V\cap W=\empt$.
Then   $F\setm U_0$ and $\overline {W}$
are disjoint closed sets. But $p_0\in X'$ implies that
$\overline {W}$ is infinite, so $F\setm U_0$ is  finite because
$X$ is SAU.
A symmetrical argument yields that $F\setm U_1$ is also finite,
hence $F=(F\setm U_0)\cup (F\setm U_1)$ implies that $F$ is finite as well.

\smallskip\noindent (3)
If $F\subs X$ is countable and closed then $F$ is compact because
$X$ is  countably compact by (1).
Consequently, $F$ is finite by (2). The second part now follows from $\overline{A} = A \cup A'$.

\smallskip\noindent (4)
First we show that
\begin{align}\label{eq:finint}
\text{$\clp(X)$ is  closed under finite intersections.}
\end{align}
Otherwise we could
choose $A,B\in \br X;{\omega};$ such that $n=|\overline A\cap \overline B| < \omega$ is minimal.
Since $X$ is SAU, we can
pick  $p\in \overline A\cap \overline B\ne \empt$.

Then $A\cup \{p\}$ is not compact by (2), hence
there is an open set $U\ni p$ such that  $A\setm U$ is infinite.
But $p\notin \overline{A\setm U}$,
so
\begin{align*}
 \overline{A\setm U}\cap \overline{B}\subs (\overline{A}\cap \overline{B})
\setm \{p\},
\end{align*}
consequently $|\overline{A\setm U}\cap \overline{B}| < n$,
which contradicts the choice of  $n$.
So we proved \eqref{eq:finint}.

Now assume that $\{F_n:n\in {\omega}\}\subs \clp(X)$.
Using  \eqref{eq:finint} and (3)
we can pick by recursion points $$p_n\in \bigcap_{m\le n}F_m\setm \{p_i:i<n\}$$ for $n<{\omega}$, and put
 $P=\{p_n:n<{\omega}\}$. Then, by (3), $P'$ is infinite, in fact even uncountable,
 and we have  $P'\subs \bigcap_{n\in {\omega}} F_n$.
\end{proof}

All our (consistent) examples of SAU spaces that we shall construct below have cardinality $\le \mathfrak{c}$.
We do not know if SAU spaces of size $> \mathfrak{c}$ can exist but we have the following related result.

\begin{theorem}
Every SAU space $X$ has a
SAU subspace of size $\le \mathfrak{c}$.
\end{theorem}

\begin{proof}
We may fix a function  $\,\varphi:\br X;{\omega};\times \br X;{\omega};\to X$
such that  $\varphi(A,B)\in \overline A\cap \overline B$ for all 
$\langle A,B \rangle \in \br X;{\omega};\times \br X;{\omega};$.
Let us also fix $Y_0\in \br X;{\omega};$. By theorem \ref{tm:sAU-basic}(3), $Y_0$ has two
accumulation points $p$ and $q$.

Since $\mf c^{\omega}=\mf c$, there is a set $Y$ with
$Y_0 \cup \{p,q\} \subs Y\in \br X;\le \mf c;$  which is $\varphi$-closed, i.e. $\varphi(A,B)\in Y$ holds
whenever $A,B\in \br Y;{\omega};$.

But then any two infinite closed subsets of $Y$ intersect,
moreover, $|Y'|>1$ because $p,q \in Y_0' \subset Y'$, hence
$Y$ is SAU.
\end{proof}

Now we turn to giving the lower and upper bound for the possible cardinalities of SAU spaces.
To do that we first prove two lemmas. The first one is purely combinatorial. To formulate it
we recall that a set $A$ is said to {\em split} another set $B$
iff both $B\cap A$ and  $B\setm A$ are infinite.
Also, a family of sets $\mc A$ is called a {\em splitting family for $X$} if
every  $B\in \br X;{\omega};$ is split by some member of $\mc A$.

\begin{lemma} \label{lm:comb}
If  $\mc A$  is a splitting family for $X$
then $|X|\le 2^{|\mc A|}$.
\end{lemma}

\begin{proof}
For every $x\in X$ let us put $\mc A(x)=\{A\in \mc A:x\in A\}$.
Then for each subfamily  $\mc B\subs \mc A$ the set $S = \{x \in X :\mc A(x)=\mc B\}$ is finite
because no element of $\mc A$ can split $S$.
Thus  the map $x\mapsto \mc A(x)$ is finite-to-one and hence we have  $|X|\le 2^{|\mc A|}$.
\end{proof}

The next lemma involves the well-known splitting number $\mathfrak{s}$ which is defined as
the smallest cardinality of a splitting family for $\omega$.

\begin{lemma}\label{lm:small-w}
If $X$ is any space of weight  $w(X)<\mf s$ then  every set $A\in \br X;{\omega};$
has an infinite subset $B$ with at most one accumulation point in $X$, i.e. such that $|B'|\le 1$.
So, if in addition, $X$ is countably compact then actually $X$ is sequentially compact.
\end{lemma}

\begin{proof}
Let $\mc U$ be a base of $X$ with  $|\mc U| <\mf s$. This implies that there is
$B\in \br A;{\omega};$ such that no element of  $\mc U$
splits $B$. But then the Hausdorff property of $X$ clearly implies $|B'|\le 1$. Indeed, if
$x$ and $y$ would be distinct accumulation points of $B$ with disjoint neighbourhoods
$U, V \in \mc U$, then both $U$ and $V$ would split $B$.

The second part follows because the countable compactness of $X$ implies $|B'| \ge 1$
for all $B\in \br X;{\omega};$.
\end{proof}

We are now ready to give the promised lower and upper bounds for the size of a SAU space.

\begin{theorem}\label{tm:bounds}
If $X$ is any SAU space then $$\mf s\le |X|\le 2^{2^{\mf c}}.$$
\end{theorem}

\begin{proof}
Assume that $|X|<\mf s$. Then there is a coarser $Hausdorff$ topology
$\varrho$ on $X$ of weight
$w(X,\varrho)\le |X|$, hence by lemma \ref{lm:small-w} there is
$B\in \br X;{\omega};$ such that $B$ has at most one accumulation point in
$(X,\varrho)$. Then $B$ has at most one accumulation point in the finer topology of $X$, as well.
But this implies that $X$ is not SAU by Theorem \ref{tm:sAU-basic}(3).

To verify  the upper bound of $|X|$,
consider any $A\in \br X;{\omega};$ and put $F=\overline A$.
Then, by Pospi\v sil's theorem, we have $|F|\le 2^{\mf c}$, hence
there is a family $\mathcal{A}$  of relatively open subsets of  $F$
that $T_2$-separates the points of $F$ and $|\mc A|\le |F| \le 2^{\mf c}$.

By Theorem \ref{tm:rc_not_star}(4),
for every  $B\in \br X;{\omega};$ we have $F \cap B' \in \clp(X)$,
hence $B$ has at least two, in fact uncountably many, accumulation points in $F$.
But then  some element of
$\mc A$  splits $B$, i.e. $\mc A$  is a splitting family for $X$.
By lemma \ref{lm:comb}, this implies
\begin{align*}
|X|\le 2^{|\mc A|}\le 2^{|F|}\le 2^{2^{\mf c}},
\end{align*}
which completes the proof.
\end{proof}

Of course, the previous results are only of interest if SAU spaces, at least consistently, exist.
So now we turn to proving that they do. Our first construction will make use of the
{\em reaping number} $\mathfrak{r}$ whose definition we recall next.

If $\mc A$ is any family of infinite sets then a set $S$ is said to reap $\mc A$
iff $S$ splits every member of $\mc A$. Now, $\mathfrak{r}$ is the minimum cardinality
of a family $\mc A \subs [\omega]^\omega$ such that no $S \in [\omega]^\omega$ reaps $\mc A$.
So the assumption $\mathfrak{r} = \mathfrak{c}$, that will figure in our construction of a SAU space given
below, is equivalent to the statement that every subfamily of $[\omega]^\omega$
of size $< \mathfrak{c}$ can be reaped by a member of $[\omega]^\omega$.

To make the inductive construction of our SAU space easier to digest,
we prove first the following lemma.

\begin{lemma}\label{lm:one-step}
Assume that $\<X,\tau\>$ is a locally countable space of weight $w(X)<\mf r$,
moreover we are given
$I,J\in \br X;{\omega};$ and a family $\{\<x_i,A_i\>:i<{\kappa}\} \subs X\times
\br X;{\omega};$, where $\kappa < \mf r$ and $x_i\in ({A_i}')^\tau$ for all $i < \kappa$.
Then, for any fixed $p \notin X$, there is a locally countable Hausdorff topology $\rho$ on $X\cup \{p\}$
such that
\begin{enumerate}[(i)]
\item $\tau\subs \rho$ and $w(\rho) \le w(\tau)$,
\smallskip
\item $p\in  (I')^\rho \cap (J')^\rho$,
\smallskip
\item $x_i \in {(A_i')}^\rho$ for all $i<{\kappa}$.
\end{enumerate}
\end{lemma}

\begin{proof}
Since $\<X,\tau\>$ is locally countable, there is a countable
$U\in \tau$ with $I\cup J\subs U$.
We fix an enumeration $U=\{u_n:n\in {\omega}\}$ and
a base $\mc B$ of $X$ of cardinality $w(\tau)$.

Next we consider the family
\begin{align}\notag
\mc C_0 = \big(\{I,J\}\cup \big\{B\cap U: B\in \mc B \big\}
\cup
\big\{A_i\cap B\cap U : i<{\kappa},\,B\in \mc B \big\}\big) \cap [U]^\omega.
\end{align}
Clearly, $|\mc C_0| < \mf r$, hence there is $D_0 \in [U \setm \{u_0\}]^\omega$ that reaps $\mc C_0$.

We continue this by recursion on $n\in {\omega}$: If $\mc C_n \in \big[[U]^\omega \big]^{< \mf r}$ and
$D_n \in [U \setm \{u_n\}]^\omega$  reaping  $\mc C_n$ have been defined, then we put
$$\mc C_{n+1}=\mc C_n\cup \{C\cap D_n\,,\, C\setm D_n: C\in \mc C_n \}\,.$$
Then $\mc C_{n+1} \in [U]^\omega$ and $|\mc C_{n+1}|= |\mc C_n| < \mf r$,
hence we may choose $D_{n+1} \in [U \setm \{u_{n+1}\}]^\omega$
that reaps $\,\mc C_{n+1}$.

Our topology $\rho$ on $X\cup\{p\}$ is generated by the family
\begin{align}
\tau\cup \big\{X\setm D_n: n\in {\omega} \big\} \cup
\big\{\,\{p\} \cup D_n \,: n\in {\omega}\big\}.
\end{align}
Then (i) and the local countability of $\rho$ are clear.
To see that $\rho$ is Hausdorff it suffices to show that any $x \in X$ and $p$
have disjoint $\rho$-neighbourhoods. But if $x = u_n$ then $X \setm D_n$ and
$\{p\} \cup D_n$ work,
while for $x \notin U$ any  $X \setm D_n$ and $\{p\} \cup D_n$ will work.

To check (ii), by symmetry, it suffices to show that $p\in  (I')^\rho$.
This is clearly equivalent to the statement
\begin{align}
 \text{ $I\cap \bigcap_{m<n}D_m$ is infinite for all $n < \omega$,}
\end{align}
which we can prove
by induction on $n$.
If $n=0$, then
$I\cap \bigcap_{m<n}D_m=I$
is infinite by assumption.
If
\begin{align}
 \text{$C=I\cap \bigcap_{m<n}D_m$ is infinite}
\end{align}
then  $C\in \mc C_n$, hence  $C\cap D_n = I\cap \bigcap_{m<n+1}D_m$
is infinite as well, because $D_n$ splits $C$.
This completes the proof of (ii).

For any $\varepsilon\in Fn({\omega},2)$ let us put
\begin{align*}
 D_\varepsilon =\bigcap_{\varepsilon(n)=1} D_n
\cap \bigcap_{\varepsilon(n)=0} (X\setm D_n).
\end{align*}
It is easy to prove by induction on $|\varepsilon|$ that for every $\varepsilon\in Fn({\omega},2)$
and for every $C \in \mc C_0$ we have $|C \cap D_\varepsilon| = \omega$.

Now, to check (iii), fix $i<{\kappa}$.
If  $x_i\in {(A_i\cap U)'}^\tau$, which certainly holds if $x_i \in U$, then for any $B \in \mc B$
with $x_i\in B$ we have $A_i\cap U\cap B \in \mc C_0$
and hence, by the above,
$A_i\cap U\cap B\cap D_\varepsilon$ is infinite
for each $\varepsilon\in Fn({\omega},2)$.
But this clearly implies that  $x_i \in (A_i \cap U)'^\rho \subs (A'_i)^\rho$.

If, on the other hand, we have  $x_i \notin {(A_i\cap U)'}^\tau$, then
$x_i\in {(A_i \setm U)'}^\tau$. But in this case  $x_i \notin U$ and then the obvious fact
that $\tau$ and $\rho$ coincide on $X \setm U$ trivially implies that
$x_i\in {(A_i \setm U)'}^\rho \subs (A'_i)^\rho$.

Thus we have verified (iii) and completed the proof of the lemme.
\end{proof}

We are now ready to formulate and prove our first existence result concerning
SAU spaces.

\begin{theorem}\label{tm:sAU-lc-cont}
If $\mf r=\mf c$ then there is a locally countable, separable, and crowded
SAU space of cardinality $\mf c$.
\end{theorem}

\begin{proof}[Proof of theorem \ref{tm:sAU-lc-cont}]
The underlying set of our space will be $\mbb Q\cup \mf c$, where $\mathbb{Q}$ is the set of rational numbers.
Our aim is to achieve that the closures of any two  members of $[\mbb Q\cup \mf c]^\omega$
intersect, so we fix an enumeration of all these pairs:
\begin{align*}
\big\{\{I_{\zeta},J_{\zeta}\}:{\zeta}< \mathfrak{c}\big\}= \big[[\mbb Q\cup \mf c]^\omega \big]^2\,,
\end{align*}
where $I_{\zeta}\cup J_{\zeta}\subs \mbb Q\cup {\zeta}$ for all ${\zeta}<\mf c$.

Then, by transfinite recursion on $\zeta \le \mf c$, we define
locally countable Hausdorff topologies
$\tau_{\zeta}$  on $X_{\zeta}=\mbb Q\cup {\zeta}$ as follows:
\begin{enumerate}[(1)]
 \item $\tau_0$
is the usual topology of $\mbb Q$;
\smallskip
\item if ${\zeta}<{\nu}\le\mf c$ then $w(\tau_\zeta) < \mf r$, moreover
\begin{align*}
\tau_{\zeta}\subs \tau_{\nu}\text{ and }
{\zeta}\in (I_{\zeta}')^{\tau_{\nu}}\cap
(J_{\zeta}')^{\tau_{\nu}};
      \end{align*}
      \smallskip
\item if ${\nu}\le \mf c$ is a limit ordinal then
$\tau_{\nu}$ is generated by
$\bigcup_{\zeta<\nu} \tau_\zeta$ on $X_\nu$;
\smallskip
\item for every ${\zeta} < \mf c$ we obtain $\tau_{\zeta + 1}$ by applying Lemma \ref{lm:one-step}
to the space $\<X_{\zeta},\tau_{\zeta}\>$ with $p={\zeta}$, the pair
$\{I_{\zeta}, J_{\zeta}\}$,
and the family
\begin{align*}
\big\{\<q,\mbb Q\setm \{q\}\>:q\in \mbb Q\big\}\cup
 \big\{\<{\xi}, I_{\xi}\> : {\xi}<{\zeta}\big\} \cup \big\{
\<{\xi}, J_{\xi}\> :{\xi}<{\zeta}\big\}.
\end{align*}
\end{enumerate}

We claim that the space $\<X_{\mf c},\,\tau_{\mf c}\>$ is as required. Indeed,
local countability and Hausdorffness is built in and,
using (4), one may easily check by transfinite induction that
$\mbb Q$ is both dense and crowded in each
$\<X_{\zeta}, \tau_\zeta\>$, hence $\<X_{\mf c},\,\tau_{\mf c}\>$ is separable and crowded.

Finally $\langle X_{\mf c}, \tau_{\mf c} \rangle$ is SAU
because for any $I,J\in \br X;{\omega};$ there is ${\zeta}<\mf c$
such that $\{I,J\} = \{I_\zeta , J_{\zeta}\}$, and so
${\zeta}\in (I')^{\tau_{\mf c}}\cap (J')^{\tau_{\mf c}} \ne \empt$.
\end{proof}

Now that we know that locally countable SAU spaces may consistently exist,
it makes sense to remark that they certainly cannot have cardinality $> 2^\mathfrak{c}$.
Indeed, this is an immediate consequence of theorem \ref{tm:bound_loc_countable_SA}
because a clopen subset of a SAU space must have finite complement.
Although we do not know if this upper bound is sharp, at least we know that it is
smaller than the upper bound for all SAU spaces given by theorem \ref{tm:bounds}.

We now turn to another method of constructing consistent examples of SAU spaces.
Unlike the construction in theorem \ref{tm:sAU-lc-cont}, this will allow us
to produce simultaneously SAU spaces of many different sizes.
These constructions will make use of certain (consistent) combinatorial principles that we
formulate below.

\begin{definition}\label{df:star}
Let ${\kappa} \ge \omega$ and ${\mu}$ be  cardinals, where ${\mu}=1$ or ${\mu}$ is infinite.
Then  $\circledast_{\kappa,\mu}$ is the following statement:
There is a sequence
\begin{align*}
 \<\<A^0_{\alpha},A^1_{\alpha}\>:{\alpha} \in {\kappa}\>
\end{align*}
such that
\begin{enumerate}[(A)]
 \item each $\<A^0_{\alpha},A^1_{\alpha}\>$ is a partition of ${\alpha}\times {\mu}$,
\item
for every $S\in \br {\kappa}\times {\mu};{\omega};$ there is
${\beta} <{\kappa}$ such that $\<\<A^0_{\alpha},A^1_{\alpha}\> : {\alpha} \in {\kappa} \setm \beta\>$
is dyadic on $S$, i.e.
\begin{align*}
 \big|S \cap \bigcap \{A^{\varepsilon({\zeta})}_{\zeta} : {\zeta}\in \dom(\varepsilon) \}\big|={\omega},
\end{align*}
whenever $\varepsilon \in Fn({\kappa}\setm {\beta},2)$.
\end{enumerate}
\end{definition}
In what follows, we shall write
$\,A[\varepsilon]=\bigcap \{A^{\varepsilon({\zeta})}_{\zeta} : {\zeta}\in \dom(\varepsilon) \}$.
Also, a sequence witnessing $\circledast_{\kappa,\mu}$ will be called
simply a $\circledast_{\kappa,\mu}$-sequence.

The following simple result is given just for orientation.

\begin{proposition}\label{*}
$\circledast_{\kappa,{\mu}}$ implies
$\mf s\le cf({\kappa})\le  \kappa\le \mf c$.
\end{proposition}

\begin{proof}
If  $\<\<A^0_{\alpha},A^1_{\alpha}\>:{\alpha}<{\kappa}\>$
is a $\circledast_{\kappa, \mu}$-sequence and $I\subs {\kappa}$ is cofinal then,
by condition \ref{df:star}(B), $\,\{A^0_{\alpha} : {\alpha}\in I\}$ is a splitting family
for $\kappa \times \mu$. Thus $\mf s\le cf({\kappa})$.

Now, assume that we had ${\kappa}>\mf c$.
Then fix $S \in [\kappa \times \mu]^\omega$ and
${\beta} <{\kappa}$ such that $\<\<A^0_{\alpha},A^1_{\alpha}\>:{\beta}\le {\alpha}<{\kappa}\>$
is dyadic on $S$. But $|{\kappa}\setm {\beta}|={\kappa}>\mf c$,
so there are ${\beta}<{\zeta}<{\xi}<{\kappa}$ such that
$A^0_{\zeta}\cap S = A^0_{\xi}\cap S$, contradicting that
$\<\<A^0_{\alpha},A^1_{\alpha}\>:{\beta}\le {\alpha}<{\kappa}\>$
is dyadic on $S$. This proves ${\kappa}\le \mf c$.
\end{proof}

To obtain SAU spaces from $\circledast_{\kappa,{\mu}}$, we actually need
$\circledast_{\kappa,{\mu}}$-sequences with an extra property that, luckily,
we can get for free.

\begin{lemma}\label{tm:strongst}
If $\circledast_{\kappa,\mu}$ holds then
there is a $\circledast_{\kappa,{\mu}}$-sequence
\begin{align*}
 \<\<A^0_{\alpha},A^1_{\alpha}\>:{\alpha}<{\kappa}\>
\end{align*}
that, in addition to \ref{df:star}(A) and \ref{df:star}(B),
satisfies  condition (C) below as well:
\begin{enumerate}[(A)]
\item[(C)] every pair $\{x,y\}\in \br {\kappa}\times {\mu};2;$ is separated by
$\<A^0_{\alpha},A^1_{\alpha}\>$ for cofinally many $\alpha < \kappa$, that is
for any   ${\beta}<{\kappa}$
there is $\alpha \in \kappa \setm \beta$ such that
\begin{align*}
 \big|\{x,y\}\cap A^0_{\alpha}\big|=1.
\end{align*}
\end{enumerate}
\end{lemma}

\begin{proof}
Let us start by fixing a $\circledast_{\kappa,{\mu}}$-sequence
\begin{align*}
 \<\<B^0_{\alpha},B^1_{\alpha}\>:{\alpha}<{\kappa}\>.
\end{align*}

Case 1: ${\mu}=1$.
We fix an injective map $f:{\kappa}\times {\kappa}\times {\kappa} \to {\kappa}$
such that $\max\{\zeta,\xi\} < f(\zeta,\xi,\eta)$
and then put
\begin{align*}
A^0_{\alpha}= \left\{
\begin{array}{ll}
(B^0_{\alpha}\cup\{{\zeta}\})\setm \{{\xi}\}&\text{if ${\alpha}=f({\zeta},{\xi},\eta)$
for some ${\zeta},{\xi},\eta<{\kappa}$,}\\\\
B^0_{\alpha}&\text{otherwise,}
\end{array}
\right .
\end{align*}
and
\begin{align*}
 A^1_{\alpha}=({\alpha}\times \{0\})\setm A^0_{\alpha}.
\end{align*}
Then $|A^i_{\alpha}\bigtriangleup B^i_{\alpha}|\le 2$ implies that property (B) is preserved,
and (C) holds because for any $\langle \zeta,\xi \rangle \in \kappa \times \kappa$ there
are cofinally many $\alpha < \kappa$ with  ${\alpha}=f({\zeta},{\xi},\eta)$.

Case 2: ${\mu}\ge {\omega}$.
For any
$a,b\in {\kappa}\times {\mu}$ we define
\begin{align*}
 a\equiv b\ \,\Leftrightarrow\,
\exists\, {\gamma < \kappa\,}\, \forall\, {\zeta}\in {\kappa}\setm {\gamma}
 \ ( a\in B^0_{\zeta} \Longleftrightarrow
b\in B^0_{\zeta} )\,.
 \end{align*}
Then  $\equiv$ is clearly an equivalence relation and property (B) implies that
every  equivalence class of $\equiv$ is finite as any
infinite subset of ${\kappa}\times {\mu}$ is split by $B^0_{\alpha}$ eventually.

This clearly implies that
if  $X\subs  {\kappa}\times {\mu}$
contains exactly one element from each $\equiv$-equivalence class
then  $ |X\cap (\{{\alpha}\}\times {\mu})|={\mu}$
for all ${\alpha}<{\kappa}$.
Pick such an $X$ and, for every $\alpha < \kappa$, fix a bijection
$$f_\alpha : \{\alpha\} \times {\mu}\to X\cap (\{{\alpha}\}\times {\mu}).$$

Now it is obvious that if we put $A^i_{\alpha}= f_\alpha^{-1}(X\cap B^{\alpha}_i)$,
then the sequence $\<\<A^0_{\alpha},A^1_{\alpha}\>:{\alpha}<{\kappa}\>$
has all the three properties (A), (B), and (C).
\end{proof}

A $\circledast_{\kappa, \mu}$-sequence with the additional property (C)
will be called a strong $\circledast_{\kappa, \mu}$-sequence.
These, as we shall now show, yield us spaces with very strong SAU properties.

\begin{theorem}\label{tm:star}
If $\<\<A^0_{\alpha},A^1_{\alpha}\>:{\alpha}<{\kappa}\>$ is a
strong $\circledast_{\kappa, \mu}$-sequence then
there is a separable and crowded Hausdorff topology $\tau$ on
${\kappa}\times {\mu}$  such that
for every infinite set $S \subs X$ there is ${\alpha}<{\kappa}$
for which  $$\overline{S} \supset (\kappa \setm {\alpha})\times {\mu} .$$
In particular, if ${\mu}<{\kappa}$, then
\begin{align}\label{eq:largeF1}
\text{ $|(\kappa \times \mu) \setm \overline S|< \kappa$ for all infinite $S\subs \kappa \times \mu\,$.}
\end{align}
\end{theorem}

\begin{proof}
Let $\rho$ be a
topology on $Q = \omega \times \{0\}$ such that $\<Q,\rho\>$ is homeomorphic to
$\mbb Q$ with its usual topology and fix  a countable base $\mc B$ of
$\<Q,\rho\>$.  Since $cf({\kappa})> {\omega}$, we can pick
${\beta} <{\kappa}$ such that, for all $B\in \mc B$,
$\<\<A^0_{\alpha},A^1_{\alpha}\>:{\beta}\le {\alpha}<{\kappa}\>$
is dyadic on $B$.

Our topology $\tau$ is generated by
\begin{align}
\mc B \cup
\{A^0_{\alpha} : {\alpha}\in {\kappa}\setm {\beta}\}
\cup
\{A^1_{\alpha} : {\alpha}\in {\kappa}\setm {\beta}\}.
\end{align}
Then $\tau$ is Hausdorff  by \ref{tm:strongst}.(C).
The choice of $\beta$ implies that $Q$ is both dense and crowded with
respect to $\tau$, hence $\tau$ is separable and crowded.

For every $S\in \br X;{\omega};$ there is
${\alpha} \in \kappa \setm \beta$ such that
$\langle \<A^0_{\zeta}, A^1_{\zeta} \rangle : {\zeta}\in {\kappa}\setm {\alpha}\>$
is dyadic on $S$. This clearly implies  $({\kappa}\setm {\alpha})\times {\mu}\subs \overline S$.
\end{proof}

The above space $X = \langle \kappa \times \mu , \tau \rangle$ trivially has the following very strong
SAU property: any family $\mathcal{A} \in [\clp (X) ]^{< cf(\kappa)}$ has non-empty intersection.

Now we turn to examining the consistency of the principles $\circledast_{\kappa, \mu}$.
As it will turn out, many instances of them hold true in a generic extension obtained by
adding a lot of Cohen reals to an arbitrary ground model. To see this, we need the following
easy but technical lemma. First we fix some notation:

\begin{definition}\label{def:Cohen}
Assume that the cardinals ${\kappa},\,{\mu}$ and ${\alpha}\in {\kappa}$ are given.
Let $\mc G$ be $Fn(({\kappa}\setm {\alpha})\times {\kappa}\times {\mu},2)$-generic over $V$.
Then in $V[\mc G]$ we put $g=\bigcup \mc G$ and define the sequence
\begin{align*}
\mc A_{\alpha}^{\mc G}=\<\<A^0_{\beta},A^1_{\beta} \>: {\alpha}\le {\beta}<{\kappa}\>
\end{align*}
by
\begin{align*}
A^i_{\beta}=\{\langle \xi,\zeta \rangle \in {\beta}\times {\mu}: g({\beta},\xi,\zeta)=i\}.
\end{align*}
\end{definition}

\begin{lemma}\label{lm:cohen}
Assume that  $\mc G$ is
$Fn(({\kappa}\setm {\alpha})\times {\kappa}\times {\mu},2)$-generic over $V$.
Then $\mc A_{\alpha}^{\mc G}$ is dyadic on each
$S\in V \cap \br {\alpha}\times {\mu};{\omega};$.
\end{lemma}

\begin{proof}
Fix $S\in V \cap \br {\alpha}\times {\mu};{\omega};$,
$\varepsilon \in Fn({\kappa}\setm {\alpha},2)$, and $T \in [S]^{<\omega}$.
For any condition $p\in Fn(({\kappa}\setm {\alpha})\times {\kappa}\times {\mu},2)$
we can find  $\<\xi,\eta\> \in S\setm T$ such that
\begin{align*}
 (\dom(\varepsilon)\times \{\<\xi,\eta\>\})\cap \dom (p)=\empt.
\end{align*}
Then we can find a condition  $q\supset p$ such that for each
${\zeta}\in \dom(\varepsilon)$ we have
 \begin{align*}
 q({\zeta}, \xi, \eta)=\varepsilon({\zeta}),
\end{align*}
hence
\begin{align*}
 q\Vdash \<\xi,\eta\>\in \mc A_{\alpha}^{\mc G}[\varepsilon]\cap (S\setm T).
\end{align*}
Since $p$  and  $T$
were arbitrary, this implies
\begin{align*}
 1 \Vdash \mc A_{\alpha}^{\mc G}[\varepsilon]\cap S\text{ is infinite},
\end{align*}
hence, as $\varepsilon$ was also arbitrary,
\begin{align*}
 1\Vdash \text{$\mc A_{\alpha}^{\mc G}$
is dyadic on $S$.}
\end{align*}
\end{proof}

\begin{theorem}\label{tm:Cohen}
If we add $\lambda > \omega$ many Cohen reals to any ground model then
in the extension $\circledast_{\omega_1,\mu}$ holds for any uncountable
$\mu \le \lambda $.
\end{theorem}

\begin{proof}
Clearly, it suffices to prove that $\circledast_{\omega_1,\lambda}$ holds in the extension.
Since $|\omega_1 \times \omega_1 \times \lambda| = \lambda$, we may assume that our generic
extension is of the form $V[G]$ where $G$ is $Fn(\omega_1 \times \omega_1 \times \lambda\,,2)$-generic
over $V$.

Now, in $V[G]$, putting  $g = \bigcup G$ we define for $\beta < \omega_1$ and $i < 2$
\begin{align*}
A^i_{\beta}=\{\langle \xi,\zeta \rangle \in {\beta}\times {\lambda}: g({\beta},\xi,\zeta)=i\},
\end{align*}
and we claim that
\begin{align*}
 \<\<A^0_{\beta},A^1_{\beta}\>:{\beta} \in {\omega_1}\>
\end{align*}
is a $\circledast_{\omega_1,\lambda}$-sequence.

Then (A) is obvious and to check (B) consider any $S \in [\omega_1 \times \lambda]^\omega$.
It is well-known, however, that there is some $\alpha < \omega_1$ such that
$S \in V[G \cap Fn(\alpha \times \omega_1 \times \lambda,2)]$. From this,
applying lemma \ref{lm:cohen} to the ground model $V[G \cap Fn(\alpha \times \omega_1 \times \lambda,2)]$,
we may immediately deduce that the tail sequence $\<\<A^0_{\beta},A^1_{\beta}\>:{\beta} \in {\omega_1} \setm \alpha\>$
is dyadic on $S$.

Actually, it is easy to see using genericity that $\<\<A^0_{\beta},A^1_{\beta}\>:{\beta} < {\omega_1}\>$
is a strong $\circledast_{\omega_1,\lambda}$-sequence.

\end{proof}

\begin{corollary}
If we add $\lambda > \omega$ many Cohen reals to our ground model then
in the extension for every cardinal $\mu \in [\omega_1, \lambda]$
there is a SAU space of size $\mu$.
\end{corollary}

Theorem \ref{tm:Cohen} and proposition \ref{*} imply the well-known and trivial fact
that $\mf s = \omega_1$ holds in a generic extension obtained by adding
uncountably many Cohen reals. Hence by the the above
corollary it is consistent to have a gap of any possible size
between $\mf s = \omega_1$ and $\mf c$ and to have
SAU spaces of all cardinalities between $\mf s$ and $\mf c$.
Our next theorem implies the consistency of the analogous statement with $\mf s > \omega_1$.

\begin{theorem}\label{tm:sAU-size}
Assume that  $GCH$ holds and ${\nu}\le {\lambda}$ are  uncountable cardinals
in our ground model $\,V$ such that $\nu$ is regular and $cf(\lambda) > \omega$.
Then there is a CCC, hence cardinal and cofinality preserving,
forcing notion $P$ such that in the extension $V^P$ we have
$\mf s={\nu}$, $\mf c = \lambda$,  moreover  $\circledast_{{\kappa},{\mu}}$ holds whenever
${\nu}\le {\kappa}=cf({\kappa})\le {\lambda}$  and either  ${\omega}\le {\mu}\le {\lambda}$
or  ${\mu}=1$.
\end{theorem}

\begin{proof}
Let $P=Fn({\lambda},2)*\dot{Q}$, where $\dot{Q}$ is a name in $V^{Fn({\lambda},2)}$ for
the standard finite support iteration
\begin{align*}
 \<Q_{\zeta}:{\zeta}\le {\lambda}, R_{\xi}:{\xi}<{\lambda} \>
\end{align*}
that forces $\mathfrak{p} = \nu$ in such a way that
each $R_{\xi}$ is a CCC, even $\sigma$-centered, poset of size $<{\nu}$
in $V^{Fn({\lambda},2)*Q_{\xi}}$.
So in $V^P$ we have $\mathfrak{c}={\lambda}$ and
$\mf p = {\nu}$.

\medskip

Let us now fix ${\kappa}$ and $\mu$ as indicated and check that $\circledast_{{\kappa},{\mu}}$ holds.
This will imply $\mf s={\nu}$ because on one hand $\nu =\mathfrak{p} \le \mathfrak{s}$,
while, by proposition \ref{*}, $\circledast_{{\nu},{\mu}}$ implies $\mathfrak{s} \le \nu$.

Because of $\kappa \cdot \mu \le \lambda$ the forcings
$Fn({\lambda},2)\times Fn({\kappa}\times {\kappa}\times {\mu},2)$ and  $Fn({\lambda},2)$
are equivalent, hence we may assume that actually
\begin{align*}
 P=Fn({\lambda},2)\times Fn({\kappa}\times {\kappa}\times {\mu},2)*\dot{Q}\,,
\end{align*}
and from now on we will work in the intermediate model
\begin{align*}
 W=V^{Fn({\lambda},2)}.
\end{align*}
Now, let $\mc G$ be  $Fn({\kappa}\times {\kappa}\times {\mu},2)$-generic
over $W$ and $\mc H$ be  $Q$-generic over $W[\mc G]$. Our result then follows from
the following claim.

\begin{claim}
The sequence
\begin{align*}
\mc A^\mc G_0=\<\<A^0_{\alpha}, A^1_{\alpha}\>:{\alpha}<{\kappa}\>\,,
\end{align*}
defined  in $W[\mc G]$ following \ref{def:Cohen} with the choice $\alpha = 0$,
is a  $\circledast_{{\kappa},{\mu}}$-sequence in the final generic extension
$W[\mc G][\mc H]$.
\end{claim}

\begin{proof}[Proof of the claim]
Let us fix $S\in \br {\kappa}\times {\mu};{\omega};$ in $W[\mc G][\mc H]$.
It is easy to see that there is
a regular suborder $Q'\lessdot  Q$ such that $|Q'|<{\kappa}$ and
\begin{align*}
S\in W[\mc G][\mc H']
\end{align*}
where $\mc H' = \mc H \cap Q'$.
Since ${\kappa}$ is  regular, then there is ${\alpha}<{\kappa}$ such that
$Q'\in W[\mc G_{\alpha}]$, where $\mc G_{\alpha} = \mc G \cap Fn({\alpha}\times {\kappa}\times {\mu},2)$.
But $W[\mc G_{\alpha}]$ contains both $Q'$
and $Fn(({\kappa}\setm {\alpha})\times {\kappa}\times {\mu},2)$,
hence, putting $$\mc G^\alpha = \mc G \cap Fn(({\kappa}\setm {\alpha})\times {\kappa}\times {\mu},2)\,,$$ we have
\begin{align}\label{eq:change}
W[\mc G][\mc H'] =
W[\mc G_{\alpha}][\mc G^{\alpha}][\mc H']=
W[\mc G_{\alpha}][\mc H'][\mc G^{\alpha}].
\end{align}

By lemma \ref{lm:cohen} then we have
\begin{align*}
W[\mc G][\mc H'] = W[\mc G_{\alpha}][\mc H'][\mc G^{\alpha}]\models
\text{``$\mc A^{\mc G^{\alpha}}_\alpha$ is dyadic on $S$''},
\end{align*}
while $\mc A^{\mc G^{\alpha}}_\alpha$ is clearly just the final segment of $\mc A^\mc G_0$
starting at $\alpha$. But then by $W[\mc G][\mc H'] \subs W[\mc G][\mc H]$ we also have

\begin{align*}
W[\mc G][\mc H] \models
\text{``$\mc A^{\mc G^{\alpha}}_\alpha$ is dyadic on $S$''},
\end{align*}
hence $\mc A^\mc G_0$ is indeed a  $\circledast_{{\kappa},{\mu}}$-sequence in $W[\mc G][\mc H]$.
\end{proof}

This completes the proof of our theorem.
\end{proof}

We have seen in theorem \ref{tm:star} that if $\mu < \kappa$ 
then our construction from $\circledast_{{\kappa},{\mu}}$ yields a space 
$X$ of size $\kappa$ with the very strong SAU property that every
infinite closed set $F \in \clp(X)$ is ``co-small'', i.e. $|X \setm F| < |X| = \kappa$.
Our following result shows that this strong property cannot be pushed any further.

\begin{theorem}\label{tm:nospace}
If $X$ is any space then
\begin{align*}
  |X|\le \sup \{|X\setm F|^+: F\in F\in \clp(X)\}.
\end{align*}
\end{theorem}

\begin{proof}
Assume, on the contrary,  that
\begin{align}\label{eq:nospace}
{\kappa}=\sup \{|X\setm F|^+: F\in F\in \clp(X)\}<  |X|.
\end{align}
Then every point $p\in X$ has an open neighborhood $U(p)$ with $|U(p)|<{\kappa}$.
But then, by Hajnal's Set Mapping
Theorem from \cite{Ha}, there is a set $Y\in \br X;|X|;$
such that $q\notin U(p)$ for any distinct $p,q \in Y$.
Fix any $Z\subs Y$ such  that $|Z|=|Y|=|X|$ and $Y \setm Z$ is infinite.
Then for $U=\bigcup_{p\in Z}U(p)$ we have $|U| = |X|$ while $X \setm U \in \clp(X)$,
contradicting (\ref{eq:nospace}).
\end{proof}

Our results below show that, consistently, the existence of a SAU space of a given size
does not imply the existence of a space of the same size in which
every infinite closed set is ``co-small''.

\begin{theorem}\label{tm:rc_not_star}
Assume that $\mf r = \mf c={\omega}_2$ and $\clubsuit$ holds.
Then SAU spaces of cardinality $\mf c$ exist but
in every space $X$ of cardinality $\mf c$ there is an infinite closed
set $F$ such that $|X \setm F| = \mf c$.
\end{theorem}

\begin{proof}
By theorem \ref{tm:sAU-lc-cont} $\mf r=\mf c$ implies that
SAU spaces of cardinality $\mf c$ exist .

Now assume that $X=\<{\omega}_2, \tau\>$ is a space
such that $|{\omega}_2\setm F|\le {\omega}_1$ for all  $F \in \clp(X)$ and
fix  a $\clubsuit$-sequence
$\<T_{\alpha}:{\alpha} \in Lim({\omega}_1)\>$.
Then
\begin{align*}
 F=\bigcap\nolimits_{{\alpha}\in Lim({\omega}_1)} \overline {T_{\alpha}}
\end{align*}
contains a final segment of ${\omega}_2$. So we may
pick two points $x,y \in F \setm \omega_1$ with disjoint open neighborhoods
$U$ and $V$, respectively. Clearly,
then $U\cap T_{\alpha}$ is infinite for all ${\alpha} \in Lim({\omega}_1)$, consequently
$|U\cap {\omega}_1|={\omega}_1$.
But $\<T_{\alpha}:{\alpha} \in Lim({\omega}_1)\>$ is a $\clubsuit$-sequence, hence
there is ${\alpha} \in Lim(\omega_1)$ with $T_{\alpha}\subs U$.
Thus we get $V\cap T_{\alpha}=\empt$,
which contradicts $y\in F \subs \overline{T_\alpha}$.
\end{proof}

The consistency of the assumptions of theorem \ref{tm:rc_not_star}
follows from a result that had been
proved by the first author back in 1983 but has never been published. So
we decided to include it here. For that we need some preparation.

For any cardinal ${\mu}$ we shall write
\begin{align*}
 S^{\omega}_{\mu}=\{{\alpha}<{\mu}: cf({\alpha})={\omega}\}.
\end{align*}

We also need the following definition.

\begin{definition}
For any given set  $X$ we define the forcing notion
$\Pjuh X=\<\pjuh X,\le \>$ as follows:
\begin{align*}
\pjuh X=\{f\in &Fn(X\times {\omega},\,2\, ;\,{\omega}_1) :\\&
\dom(f)=A\times n \text{ for some }
A\in \br X;{\le \omega}; \text{ and }
n\in {\omega}\}.
\end{align*}
For $p,q\in \pjuh X$ we let $p\le q$ iff $p\supset q$.
\end{definition}

We now present some properties of this forcing.

\begin{theorem}\label{tm:j1}
Let ${\kappa}$ be any infinite cardinal in our ground model $V$.
\begin{enumerate}[(1)]
\item $\Pjuh \kappa$ is $\mathfrak{c}^+$-CC; in fact, for any $\{p_{\alpha}:{\alpha}<\mathfrak{c}^+\}\subs \pjuh{{\kappa}}$  there is
$I\in \br {\mathfrak{c}^+};{\mathfrak{c}^+};$ such that $\,\bigcup_{{\alpha}\in K}p_{\alpha}\in \pjuh{{\kappa}}$ whenever $K\in \br I;{\omega};$.
Consequently, the forcing $\Pjuh \kappa$ preserves all cardinals $> \mathfrak{c}$.
\smallskip
 \item $\mathfrak{c}$ becomes countable in $V^{\Pjuh{{\kappa}}}$, hence $(\mathfrak{c}^+)^V = (\omega_1)^{V^{\Pjuh{\kappa}}}$.
\smallskip 
\item If
$\clubsuit(S^{\omega}_{\mathfrak{c}^+})$ holds in $V$  then
$\clubsuit$ holds in $V^{\Pjuh{\kappa}}$ .
\smallskip
\item If ${\kappa}={\kappa}^\mathfrak{c}$ then  $\mathfrak{c}^{V^{\Pjuh{\kappa}}}={\kappa}$ and
$\,V^{\Pjuh{\kappa}}\models \, \MAC\,$.

\end{enumerate}
 \end{theorem}

\begin{proof}
(1) Assume that $\dom(p_{\alpha})=A_{\alpha}\times n_{\alpha}$ for
${\alpha}< \mathfrak{c}^+.$ Clearly we can find $H \in \br {\mathfrak{c}^+};{\mathfrak{c}^+};$
and $n\in {\omega}$  such that $n_{\alpha}=n$ for all ${\alpha}\in H$. A simple $\Delta$-system
and counting argument then yields $I\in \br H;\mathfrak{c}^+;$ such that  the functions
$\{p_{\alpha}:{\alpha}\in I\}$
are pairwise compatible.
It is obvious then that $I$ is as required.

\smallskip
(2)  Let $\mc G$ be  $\Pjuh {\kappa}$-generic over $V$, then
$g=\bigcup \mc G : {\kappa}\times {\omega}\to 2$. For each $n \in \omega\,$
we define the function $d_n \in {}^{\omega}2\,$ by putting for all $i < n$
\begin{align*}
 d_n(i)=g(i,n).
\end{align*}
It is straight forward to check that if
$r:{\omega}\to 2$ is in the ground model then 
\begin{align*}
 D_r=\{p\in \pjuh{{\kappa}} : \exists n\in {\omega}\, \forall i\in {\omega} \, [r(i)= p(i,n)]\}.
\end{align*}
is  dense in $\Pjuh{{\kappa}}$, consequently we have
\begin{align*}
 V^{\Pjuh {\kappa}}\models \{d_n:n<{\omega}\} \supset {}^{\omega}2\cap V.
\end{align*}
But this clearly implies that $\mathfrak{c}$ becomes countable in $V^{\Pjuh{{\kappa}}}$.
Then $(\mathfrak{c}^+)^V = (\omega_1)^{V^{\Pjuh{\kappa}}}$ follows because $(\mathfrak{c}^+)^V$
remains a cardinal by (1).

\smallskip
(3)  To aid readability, we write $\mu = \mathfrak{c}^+$ and $S = S^{\omega}_{\mu}$. 
Then we fix a $\clubsuit(S)$-sequence 
$\<A_{\zeta}:{\zeta} \in S \>$ in $V$. By (2) we have  $S \subset Lim(\omega_1)$
in the generic extension $V^{\Pjuh{\kappa}}$.

Let us assume now that $p \Vdash \dot X \in [\mu]^\mu$ for a condition $p$ in ${\Pjuh{\kappa}}$.
We can then define in $V$ a strictly increasing map $\varphi : S \to \mu$ and for each $\zeta \in S$
a condition $p_\zeta \le p$ such that $p_\zeta \Vdash \varphi(\zeta) \in \dot X$.
Applying (1)  we can find
$I\in \br {\mu};{\mu};$ such that 
$\,p_K = \bigcup_{{\zeta}\in K}p_{\zeta}\in \pjuh{{\kappa}}$ holds whenever $K\in \br I;{\omega};$.

Now, $\varphi[I] = \{\varphi(\zeta) : \zeta \in I\} \in [\mu]^\mu$, hence there is some
$\eta \in S$ such that $A_\eta \subset \phi[I]$. But then for
$K= \phi^{-1}[A_{\eta}] \in [I]^\omega$
we have $p_K \in \pjuh{{\kappa}}$ and $p_K \le p$, moreover 
we clearly have $\,p_K \Vdash A_{\eta}\subs \dot X$. Thus, no matter how we define
$A_\zeta$ for $\zeta \in Lim(\omega_1) \setm S$, the sequence $\<A_{\zeta}:{\zeta} \in Lim(\omega_1) \>$ 
will be a $\clubsuit$-sequence in the generic extension $V^{\Pjuh{\kappa}}$.
\smallskip

(4)  For each ${\alpha}<{\kappa}$ we define the real $\,q_{\alpha}\in {}^{\omega}2$
in  $V^{\Pjuh{\kappa}}$ by stipulating
$q_{\alpha}(n)=g({\alpha},n)$ for all $n \in \omega$.  Then, by genericity,
$\{q_{\alpha} : {\alpha}<{\kappa}\}$ are pairwise distinct, hence we have
$\mathfrak{c}^{  {V^{{\Pjuh{\kappa}}}}} \ge  {\kappa}$.
On the other hand, by (1) $\,\Pjuh {\kappa}$ satisfies the $\mathfrak{c}^+$-chain condition,
hence the standard calculation using nice names and the condition 
${\kappa}={\kappa}^\mathfrak{c}$ yield us that
$\mathfrak{c}^{  {V^{{\Pjuh{\kappa}}}}}\le  {\kappa}$.
Thus indeed $\mathfrak{c}^{  {V^{{\Pjuh{\kappa}}}}} = {\kappa}$.

Now suppose that $\mathfrak{c}^V \le {\lambda}<{\kappa}$ and $\mc D=\{D_{\alpha}:{\alpha}<{\lambda}\}$
is a family of dense subsets of $Fn({\omega},2)$ in  $V^{\Pjuh{\kappa}}$. 
Then there is $I\in \br {\kappa};{\lambda};$
such that $\mc D\in V^{\Pjuh{I}}$.
Pick any ${\alpha}\in {\kappa}\setm I$. Then, as $\mc D\in V^{{\Pjuh{{\kappa}\setm \{{\alpha}\}}}}$ and
$${\Pjuh{{\kappa}}}\approx{\Pjuh{{\kappa}\setm \{{\alpha}\}}} \times Fn({\omega},2), $$
$q_{\alpha}$ is generic over $\mc D$.
This clearly implies \MAC\ in the generic extension $V^{\Pjuh{\kappa}}$.
\end{proof}

In the constructible universe $L$ we have $\mathfrak{c} = \omega_1$, $(\omega_2)^{\omega_1} = \omega_2$,
moreover $\clubsuit(S^{\omega}_{\mathfrak{c}^+})$ holds. Also, it is well-known and
easy to prove that \MAC\ implies $\mathfrak{r} = \mathfrak{c}$. Consequently, it is an
immediate corollary of theorem \ref{tm:j1} that $L^{\Pjuh{\omega_2}}$ satisfies all the assumptions of
theorem \ref{tm:rc_not_star}.

\section{Problems}

In this section we formulate the most intriguing questions concerning SAU spaces
that are left open.

\begin{problem}
 Is there a SAU space in ZFC?
\end{problem}

\begin{problem}
Is it consistent that there is a SAU  space
of cardinality $>\mf c$? Is it consistent that there is a locally countable SAU  space
of cardinality $>\mf c$? 
\end{problem}

\begin{problem}
Does the existence of a SAU space imply the existence of a crowded SAU space?
\end{problem}

\end{document}